\documentclass[12pt]{amsart}
\usepackage{url, mathrsfs}
\usepackage{amssymb}
\usepackage{hyperref}
\usepackage{amsmath}
\usepackage{graphicx, epstopdf}
\usepackage{subfig}
\usepackage{color}
\usepackage{bbm}
\usepackage{subfloat}
\usepackage[draft]{fixme}
\usepackage{bm} 
\usepackage{bbm} 
\usepackage{epigraph}
\usepackage{hyperref}
\usepackage{endnotes}
\usepackage{ifpdf}
\usepackage{array}
\usepackage{wrapfig}
\usepackage{longtable}
\usepackage[framemethod=TikZ]{mdframed}
\usepackage{lipsum}
\usepackage{algorithm}
\usepackage[noend]{algpseudocode}

\usepackage{multicol}
\usepackage{multirow}
\usepackage{graphicx}
\usepackage{tikz}
\usepackage{tikz-cd}
\usepackage{makecell}

\newcommand{\R}{\mathbb{R}}
\newcommand{\F}{\mathbb{F}}

\newcommand{\C}{\mathbb{C}}

\newcommand{\D}{\mathscr{D}}

\theoremstyle{plain}

\newtheorem{thm}{Theorem}[section]

\newtheorem{cor}[thm]{Corollary}
\newtheorem{lem}[thm]{Lemma}
\newtheorem*{question}{Question}

\newtheorem*{claim*}{Claim}
\newtheorem*{thm*}{Theorem} 
\newtheorem*{cor*}{Corollary}

\newtheoremstyle{case}{}{}{}{}{}{:}{ }{}
\theoremstyle{case}

\theoremstyle{definition}

\newtheorem{exampleth}[thm]{Example}
\newenvironment{example}{\begin{exampleth}}{\hfill
    $\diamond$\\ \end{exampleth}}

\newtheorem{remark}[thm]{Remark}

\newtheorem{main}{Main Result}
\newtheorem*{exampleth*}{Example}
\newenvironment{example*}{\begin{exampleth*}}{\hfill
    $\diamond$ \end{exampleth*}}

\DeclareMathOperator{\SL}{SL}
\DeclareMathOperator{\diag}{diag}

\usepackage[inner=1in,outer=1in,top=1in,bottom=1in]{geometry}

\begin{document}
	\nocite{*}
	\title{The Fiber of the principal minor map}
	
	\author{Abeer Al Ahmadieh}
	\address{School of Mathematics, Georgia Institute of Technology, Atlanta, USA 30332} 
	\email{aahamadieh3@gatech.edu}

	\begin{abstract} 
	This paper explores the fibers of the principal minor map over a general field. The principal minor map is the map that assigns to each $n\times n$ matrix the $2^n$-vector of its principal minors. In $1984$, Hartfiel and Loewy proposed a condition that was sufficient to ensure that the fiber of the principal minor map is a single point up to diagonal equivalence. Loewy later improved upon this condition in $1986$. In this paper, we provide a necessary and sufficient condition for the fiber to be a point up to diagonal equivalence. Additionally, we establish a connection between the reducibility of a matrix and the reducibility of its determinantal representation. Using this connection, we fully characterize the fiber of symmetric and Hermitian matrices in the space of $n\times n$ matrices over any field $\F$. We also use these techniques to answer a question of Borcea, Br\"and\'en, and Liggett concerning real stable matrices.	
	\end{abstract}
	\maketitle

	\section{Introduction}\label{sec:Intro}
	The primary objective of this paper is to investigate the relationship between two matrices with equal principal minors over a field $\F$. It is easy to see that if two matrices are diagonally equivalent, then their principal minors are equal. Two matrices $A$ and $B$ are \textit{diagonally equivalent} if there exists an invertible diagonal matrix $D$ such that $A = DBD^{-1}$ or $A=DB^T D^{-1}$. The intriguing question is to determine the conditions under which this remains the sole connection between the two matrices. Reducible matrices, for instance, might have equal principal minors without being diagonally equivalent. A matrix is \textit{reducible} if it can be written as a block upper triangular matrix after permuting some rows and the corresponding columns. If we replace each nonzero entry with one and the diagonal entries with zero, then a reducible matrix corresponds to the adjacency matrix of a directed graph with at least two strongly connected components. The block diagonal matrix $\tiny{\begin{pmatrix}
		E & 0 \\ 0 & G
	\end{pmatrix}}$, for instance, has principal minors equal to those of all matrices of the form $\tiny{\begin{pmatrix}
		E & F \\ 0 & G
	\end{pmatrix}}$.\\

	Another motivating goal of this paper is to characterize \textit{real stable matrices} and answer a question of Borcea, Br\"and\'en, and Liggett \cite[Question~3.4]{BBL09}. An $n \times n$ complex matrix $A$ is real stable if its \textit{determinantal polynomial} $f_A = \det(A + \diag\{x_1,\hdots,x_n\})$ is real stable. A polynomial $f$ is \textit{real stable} if it has no zeros with strictly positive imaginary parts. 	Real stable polynomials have many applications in different fields, see for instance \cite{AOR,ConvAnalysis,BBL09,GurvitsPerm, RamGraph}. In \cite{AV22}, the author and Vinzant prove that if $f_A$ is real stable for some complex matrix $A$, then there exists a Hermitian matrix $H$ such that $f_A = \det(H+\diag\{x_1,\hdots,x_n\})$. Since the coefficients of $f_A$ are the principal minors of $A$, the matrices $H$ and $A$ share the same principal minors. Therefore, to characterize real stable matrices, we study the structure of matrices that have principal minors equal to that of a Hermitian matrix. \\
	
	Another way of rephrasing the same problem is by considering the \textit{principal minor map}. Given an $n \times n$ matrix $A$ with entries in a field $\F$, let $A_S$ denote the principal minor indexed by $S$, that is the determinant of the submatrix of $A$ indexed by the set $S$ on the rows and columns. The principal minor map is the map that assigns to each matrix the vector of its principal minors, namely
	\[
	\varphi: \F^{n\times n} \longrightarrow \F^{2^n} \ \ \text{ given by } \ \ A \longrightarrow \left(A_S\right)_{S\subseteq[n]}.
	\]	
	where we set $A_\emptyset = 1$. In this paper we are interested in studying the preimage of this map. This problem dates back to $1980$ when Engel and Schneider \cite{EngelSchneider} proved that a completely reducible matrix $A$ that belongs to the fiber of a symmetric matrix $B$ must be diagonally equivalent to it. In $1984$, Hartfiel and Loewy \cite{HartfielLoewy} considered the fibers of general matrices over any field $\F$. They proved that if the fiber contains an irreducible matrix $A$ with the property that the submatrices $A[X,X^c]$ and $A[X^c,X]$ have rank at least two for any subset $X$ of $[n]$ with $2 \leq |X| \leq n-2$, then the fiber is a point up to diagonal equivalence. Here $X^c$ denotes the complement of $X$ in $[n]$. This result was improved in $1986$ by Loewy \cite{Loewy}, who proved that if only one of the submatrices $A[X,X^c]$ or $A[X^c,X]$ of an irreducible matrix $A$ has rank at least two, then the fiber of $A$ is still a point up to diagonal equivalence. In both papers, the authors consider the inverse of the matrix $\left(\diag\{x_1,\hdots,x_n\} + A\right)$. We say that a matrix $A$ has a \textbf{cut} $X\subset [n]$ if $2\leq |X| \leq n-2$ and the ranks of $A[X,X^c]$ and $A[X^c,X]$ are at most one.
	
	In this paper we consider the fiber of any $n\times n$ matrix over a field $\F$. While Loewy \cite{Loewy} proves that an irreducible matrix with no \textit{cuts} have only one preimage up to diagonal equivalence, we show that such matrices and irreducible symmetric matrices are the only matrices that satisfy this property. 
	\begin{main}[Theorem~\ref{thm:Main}]
		The fiber of a matrix over any field $\F$ under the principal minor map consists of only one point up to diagonal equivalence if and only if the matrix is irreducible and either it has no cuts or it is diagonally equivalent to a symmetric matrix.
	\end{main}
	Using this result, it is straightforward to check on a computer whether the fiber of a matrix is a point up to diagonal equivalence. If that is not the case, then using the proof of the theorem, one can generate an algorithm to produce another point in the fiber.
	
	By relating the structure of the matrix $A$ to the factorization of its \textit{determinantal polynomial} $f_A$
	\[
	f_A = \det(\diag\{x_1,\hdots,x_n\}+A), 
	\]
	we characterize the fibers of symmetric and Hermitian matrices in the space of $n\times n$ matrices. In particular, we prove that a matrix $A$ can be written as a block upper triangular matrix with $s$ diagonal blocks if and only if its determinantal polynomial $f_A$ factors into $s$ irreducible factors. A direct consequence is that  the fiber of an irreducible matrix under the principal minor map consists only of irreducible matrices. Using this relation we characterize the fiber of a symmetric matrix under the principal minor map in the space of $n\times n$ matrices over any field $\F$. Notice that a symmetric matrix is either irreducible or completely reducible. A \textit{completely reducible} matrix is a matrix that can be written as a block diagonal matrix by permuting some rows and the corresponding columns, with each diagonal block being irreducible.
	\begin{main}[Theorem~\ref{thm:FiberSym}]
		Let $A$ be a symmetric matrix with entries in a field $\F$. Then either
		\begin{enumerate}
			\item The matrix $A$ is irreducible and the fiber of $A$ consists of all matrices that are diagonally equivalent to $A$, or
			\item The matrix $A$ can be written as a completely reducible matrix with $s$ diagonal blocks $A_1,\hdots,A_s$ that are irreducible, and the fiber of $A$ in $\F^{n\times n}$ consists of all matrices that can be written as block upper triangle matrices with $s$ irreducible diagonal blocks that are diagonally equivalent to the diagonal blocks of $A$.
		\end{enumerate} 
	\end{main}
	In particular, this theorem implies that the fiber of an irreducible symmetric matrix with entries in a field $\F$ consists of a single point up to diagonal equivalence. Notice also that the result of Engel and Schneider \cite[Theorem $3.5$]{EngelSchneider} follows from this theorem. Moreover, one can conclude that the fiber of a symmetric matrix $A$ in the space of symmetric matrices, $\rm{Sym}_n(\F)$, consists of matrices that are diagonally equivalent to $A$ via a diagonal matrix with entries equal $\pm1$ (see for instance \cite{RisingKuleszaTaskar15}) from the above theorem. 
	
	We also investigate the fibers of Hermitian matrices over the complex field. While the fiber of an irreducible symmetric matrix consists of a single point, the fiber of an irreducible Hermitian matrix can be bigger. An example illustrating this is provided in \cite[Example 4.8]{AV22}. However, we establish that the fiber consists solely of irreducible Hermitian matrices, up to diagonal equivalence. This characterization, combined with the previous argument regarding the structure of a matrix based on its determinantal polynomial factorization, grants us a comprehensive understanding of the fiber of a Hermitian matrix. Similar to the case of symmetric matrices, a reducible Hermitian matrix can be expressed as a completely reducible matrix.
		
	\begin{main}[Theorem~\ref{thm:FiberHer}]
			Let $A$ be a Hermitian matrix. Then either
			\begin{enumerate}
				\item The matrix $A$ is irreducible and the fiber of $A$ consists of Hermitian irreducible matrices up to diagonal equivalence, or
				\item The matrix $A$ can be written as a completely reducible matrix with $s$ diagonal blocks $A_1,\hdots,A_s$ that are irreducible and the fiber of $A$ consists exactly of matrices that can be written as block upper triangular matrices with $s$ diagonal blocks that are irreducible and Hermitian up to diagonal equivalence and that have principal minors equal to those of one of the $A_i$s.
			\end{enumerate} 
	\end{main}
	This characterization along with the author's work with Vinzant \cite[Theorem $6.4$]{AV22} answer a question of Borcea, Br\"and\'en, and Liggett:
	\begin{question}\label{ques:StableMatrices}\cite[Question $3.4$]{BBL09}
		Characterize the class of all real stable $n\times n$ matrices, that is, $n \times n$ matrices $A$ such that $\det(A + \diag\{x_1,\hdots,x_n\})$ is real stable.
	\end{question}
	\begin{main}[Corollary~\ref{cor:StableMatrix}]
		A matrix $A$ is real stable if it can be written as a block upper triangular matrix with diagonal blocks that are irreducible and Hermitian, up to diagonal equivalence, and that have determinantal polynomials equal to the irreducible factors of the polynomial $f_A = \det(\diag\{x_1,\hdots,x_n\}+A)$.
	\end{main}
	
	The fibers of skew symmetric matrices were addressed by Boussa\"{i}ria and Cherguia \cite{SkewSym}. They prove that if two skew symmetric matrices have equal principal minors of size at most $4$, and one of the matrices is dense and does not contain any cut, then the matrices are diagonally equivalent by a diagonal matrix with diagonal entries equal $\pm1$.\\

	While in this paper we focus on the preimage of the principal minor map, the image of this map has been studied in various contexts. Holtz and Sturmfels \cite{HoltzSturmfels07} explored the image of the space of real and complex symmetric matrices under the principal minor map. They demonstrated that the image is both closed and invariant under the action of the group $\SL_2(\R)^n\rtimes S_n$. They conjectured that the vanishing of polynomials in the orbit of the hyperdeterminant, under the action of this group, defines the image of the principal minor map over the field of complex numbers. This conjecture was resolved by Oeding \cite{Oeding11}. In a recent work by the author and Vinzant \cite{AV21}, techniques from a previous study by Kummer, Plaumann, and Vinzant \cite{KPV15} were utilized to generalize this result, making it applicable to arbitrary unique factorization domains. Moreover in \cite{AV22}, the author and Vinzant prove that the image of the space of Hermitian matrices, under the principal minor map, is cut out by the orbit of finitely many equations and inequalities under the action of $\SL_2(\R)^n\rtimes S_n$. We also examined such representations over more general fields with quadratic extensions. However, in contrast to the Hermitian case, it was shown that for any field $\F$, there is no finite set of equations whose orbit under $\SL_2(\R)^n\rtimes S_n$ defines the image of $n \times n$ matrices over $\F$ under the principal minor map for all values of $n$.

	This paper is organized as follows. In Section~\ref{sec:Background}, we introduce terminology and the basic notations that we use throughout the paper. In Section~\ref{sec:Laplace} we state and prove the generalized Laplace theorem, and we use it in Section~\ref{sec:Main} to construct a matrix in the fiber of an irreducible matrix with a cut and consequently characterize matrices with a single point fiber. In Section~\ref{sec:StructureThm}, we give a description of the structure of a matrix $A$ according to its determinantal polynomial. This description gives a characterization of the fibers of symmetric matrices in Section~\ref{sec:FiberSym} and of Hermitian matrices in Section~\ref{sec:FiberHer}, where a description of stable matrices is also provided. \\

 {\bf Acknowledgements}. We are grateful to Cynthia Vinzant for the helpful comments and discussions. We also thank Jonathan Leake and Josephine Yu for their helpful comments on the paper.

\section{Background and notation}\label{sec:Background}

Through out this paper we use $\F^{n\times n}$ to denote the set of $n\times n$ matrices with entries in the field $\F$. Let $\D(n)$ denote that set of all matrices $A$ such that the fiber of $A$ consists of matrices that are \textit{diagonally equivalent} to $A$ or $A^T$, where $A^T$ denotes the transpose of the matrix $A$. For simplicity, we use the expression \textit{diagonally equivalent} by which we mean diagonally equivalent to the matrix or to its transpose. One of the goals of this paper is to characterize the set $\mathscr{D}(n)$. 

We denote by $A^{\rm adj}$ the adjugate of the matrix $A$ and by ${\rm rk}(A)$ the rank of the matrix $A$. The \textit{determinantal polynomial} of a matrix $A$ is the polynomial $f_A = \det(\diag\{x_1,\hdots,x_n\}+A)$. Notice that a determinantal polynomial is multiaffine. A polynomial is \textit{multiaffine} if it has degree at most one in each variable.
\subsection{The Rayleigh differece of a polynomial}
The Rayleigh difference of a polynomial $f$ with respect to two variables $x_i$ and $x_j$ is given by
\[
\Delta_{ij}(f)  = \frac{\partial f}{\partial x_i}  \frac{\partial f}{\partial x_j} - f \frac{\partial^2 f}{\partial x_i \partial x_j}.
\]
In \cite{AV21} and \cite{AV22}, the author and Vinzant use the Rayleigh difference of a polynomial in order to characterize multiaffine determinantal polynomials of symmetric and Hermitian matrices. We also use them to characterize determinantal polynomials of $n\times n$ matrices as the following theorem shows.

\begin{thm}[Theorem~3.1 in \cite{AV22}]\label{thm:AV}
	Let $f \in \F[x_1,\hdots,x_n]$ be
	multiaffine in the variables $x_1,\hdots,x_n$ with its coefficient of $x_1\cdots x_n$ equals one.
	Then $f= \det({\rm diag}(x_1,\hdots,x_n) +A)$ for some $A\in \F^{n\times n}$ if and only if for every $i\neq j\in [n]$, the polynomials $\Delta_{ij}(f)$ factor as the product $g_{ij}\cdot g_{ji}$ where 
	\begin{itemize}
		\item[(a)] $g_{ij}\in  \F[x_k : k\neq i,j]$ is multiaffine in $x_1,\hdots, x_n$ and  
		\item[(b)] for every $k\in [n]\backslash \{i,j\}$, ${\rm res}_{x_k}(g_{ij}, f) = g_{ik}g_{kj}$. 
	\end{itemize}
	In this case, we can take $g_{ij}$ to be the $(i,j)$th entry of $({\rm diag}(x_1,\hdots,x_n) +A)^{\rm adj}$, with $M^{\rm adj}$ represents the adjugate matrix of $M$.
\end{thm}
We use ${\rm res}_{x_k}(g, h)$ to denote the \textit{resultant} of two polynomials $g$ and $h$ with respect to a variable $x_k$:
\[
{\rm res}_{x_k}(g,h) = (g|_{x_k=0}) \cdot \frac{\partial}{\partial x_k} h - (h|_{x_k=0}) \cdot \frac{\partial}{\partial x_k}g. 
\]
Moreover, we use the Rayleigh differences to study the irreducibility of a multiaffine polynomial as shown in the following lemma from \cite{AV22}. We include the proof here for completeness.
\begin{lem}[Proposition~2.2 in \cite{AV22}]\label{lem:ZeroDelta}
	If $f\in \F[x_1, \hdots, x_n]$ has degree one in each of $x_i$ and $x_j$, 
	then $\Delta_{ij}(f)=0$ if and only if $f$ factors into polynomial $g\cdot h$ with 
	$g\in \F[x_k : k\neq i]$ and $h\in \F[x_k:k\neq j]$.
\end{lem}
\begin{proof}
	By assumption we can write $f = ax_ix_j + b x_i + c x_j + d$ for $a,b,c,d\in \F[x_k: k\neq i,j]$. 
	Then $\Delta_{ij}(f) = bc-ad$. If $\Delta_{ij}(f) = 0$, then there is some factorization 
	$b = b_1 b_2$ and $c = c_1c_2$ for which $a = b_1c_1$ and $d = b_2c_2$.  
	Hence $f = (b_1 x_i + c_2)(c_1x_j + b_2)$.  Similarly, if $f = (b_1 x_i + c_2)(c_1x_j + b_2)$ for some $b_1, b_2, c_1, c_2\in \F[x_k: k\neq i,j]$, then $\Delta_{ij}(f) = bc-ad=0$. 
\end{proof}
\begin{thm}\label{thm:Dodgson}
	Let $A$ be an $n\times n$ matrix in $\F^{n\times n}$. Let $f = \det(\diag\{x_1,\hdots,x_n\}+A)$ and $G= (\diag\{x_1,\hdots,x_n\}+A)^{\rm adj}$. Then
	\[
	\Delta_{ij}(f) = G_{ij} G_{ji},
	\]
	where $G_{ij}$ denotes the $(i,j)^{\rm th}$ entry of $G$.
\end{thm}
\begin{proof}
	Let $S, T\subset [n]$ be of equal cardinality. Let $M(S,T)$ denote the submatrix of $M$ obtained by dropping rows $S$ and columns $T$ from $M$. Then for any $i\neq j\in [n]$, 
	\begin{equation*}
		\det(M(i,i))\cdot \det(M(j,j))  - \det(M)\det(M(\{i,j\}, \{i,j\})) = \det(M(i,j))\cdot \det(M(j,i)).
	\end{equation*}
	This follows from a classical equality used by Dodgson \cite{Dodgson} as a method for computing determinants. Since $M = {\rm diag}(x_1,\hdots,x_n) + A$, then for any subset $S\subseteq [n]$, the principal minor $\det(M(S,S))$ equals the derivative of $f$ with respect to the variables in $S$,  
	$\left(\prod_{i\in S}\frac{\partial}{\partial x_i}\right) f$. Since $G_{ij} = \det(M(i,j))$, the equation above then gives that $\Delta_{ij}(f)$ equals $G_{ij} G_{ji}$. 
\end{proof}
\section{The Generalized Laplace Expansion}\label{sec:Laplace}
In this section we give a complete proof of the \textit{generalized Laplace theorem}. While the Laplace expansion theorem states that the determinant of an $n\times n$ matrix $A$ over any field $\F$ can be computed by summing the products of the elements of some row or column with their corresponding cofactors, or $n-1\times n-1$ minors, the generalized Laplace expansion \cite{Laplace} computes the determinant by summing up minors of size $k\times k$ and their corresponding $n-k \times n-k$ minors. In this section we state the theorem and we provide a proof based on exterior algebra for completeness.
\begin{thm}\label{thm:Laplace}[The Generalized Laplace Theorem]
	Let $A$ be an $n\times n$ matrix with entries in a field $\F$. Let $S \subset [n]$ such that $|S| = k$ with $1 \leq k \leq n-1$. The determinant of $A$ is given by:
	\[
	{\rm det}(A) = \sum_{\substack{T\subset [n]\\ |T| = k}} (-1)^{\sum T+\sum S}A_{S,T} A_{S^c,T^c}
	\]
	where $\sum T = \sum_{t\in T} t$.
\end{thm}
To prove the theorem we use the following lemma.
\begin{lem}\label{lem:SignofPer}
	Let $S$ and $T$ be two subsets of $[n]$ of size $k$. Let $\sigma$ be the permutation of $[n]$ given in two line notation by
	\[
	\sigma = \begin{pmatrix}
		S & S^c \\ T & T^c
	\end{pmatrix}
	\]
	where the elements of $S, T, S^c,$ and $T^c$ are arranged in increasing order. Then 
	\[
	\rm{sgn}(\sigma) = (-1)^{\sum S+ \sum T}.
	\]
\end{lem}

\begin{proof}
	First assume that $S = \{1,2,\ldots,k\}$. We will proceed by induction on $k$. Assume that $k=1$ and the smallest element of $T$ is $t_1$. Since the elements of $T^c$ are listed in increasing order, then the sign of $\sigma$ is $(-1)^{t_1-1}=(-1)^{t_1+1}$. For the inductive step, $t_1$ should be transposed $t_1-1+(k-1)$ times, and so ${\rm sgn}(\sigma) = (-1)^{t_1+k-2} {\rm sgn}(\tau)$ where 
	\[
	\tau = \tiny{\begin{pmatrix}
			1, \ldots, k-1 & k, \ldots, n \\ T_1 & (T_1)^c    
	\end{pmatrix}}.
	\] 
	Here $T_1$ denotes the set $T\setminus \{t_1\}$, and the elements of $T_1$ and $T_1^c$ are arranged in increasing order. Applying the inductive hypothesis to $\tau$, we infer that ${\rm sgn}(\tau) = (-1)^{\sum T_1 + \sum_{i=1}^{k-1}i}$. Therefore, 
	\[
	{\rm sgn}(\sigma)=(-1)^{t_1+k-2}(-1)^{T_1+\sum_{i=1}^{k-1}i} = (-1)^{\sum T + \sum_{i=1}^k i}.
	\]  
	For general $S$, $\sigma$ can be written as the composite of two permutations $ \sigma_1 $ and $\sigma_2$ given in two line notation by $
	\sigma_1 = \tiny{\begin{pmatrix}
			S & S^c \\ 1,\ldots,k & k+1,\ldots,n
	\end{pmatrix}}$ and $\sigma_2 = \tiny{\begin{pmatrix}
			1,\ldots,k & k+1,\ldots,n \\ T & T^c
	\end{pmatrix}}$. Thus, the sign of $\sigma $ is the product of the signs of $\sigma_1$ and $\sigma_2$. Since a permutation has the same sign as its inverse, we conclude that $ {\rm sgn}(\sigma)= (-1)^{\sum T + \sum_{i=1}^k i} (-1)^{\sum S + \sum_{i=1}^k i} = (-1)^{\sum T + \sum S}$ as desired.
\end{proof}
\begin{proof}[Proof of Theorem~\ref{thm:Laplace}]
	For simplicity, we assume $S = [k]$. We denote by $c_i$ the $i^{{\rm th}}$ column of the matrix $A$ and by $e_i$ the vector in $\F^n$ with $i^{\rm th}$ entry equal to one and zero otherwise. Then, 
	\[
	c_1 \wedge c_2 \wedge \cdots \wedge c_n= {\rm det}(A) e_1\wedge \cdots \wedge e_n  
	\] 
	Let $c_i = \sum_{j=1}^n a_{ji} e_j$. Then the wedge product of the first $k$ columns can be written as
	\[
	c_1 \wedge \cdots \wedge c_k = \sum_{j=1}^n a_{j1}e_j \wedge \cdots \wedge \sum_{j=1}^n a_{jn}e_j = \sum_{\substack{T\subset[n]\\|T| = k}} (\sum_{j\in T} a_{j1}e_j \wedge \cdots \wedge \sum_{j\in T} a_{jn}e_k) = \sum_{\substack{T\subset[n]\\ |T| = k}} A_{T,S} e_T
	\]
	where $e_T = e_{t_1}\wedge \cdots \wedge e_{t_k}$ and $t_1,\ldots,t_k$ are the elements in $T$ arranged in increasing order.
	Therefore
	\[
	(c_1\wedge \cdots\wedge c_k)\wedge(c_{k+1} \wedge \cdots \wedge c_{n}) = \sum_{\substack{T\subset[n]\\ |T| = k}} A_{T,S} e_T \wedge \sum_{\substack{T\subset[n]\\ |T| = k}} A_{T^c,S^c} e_{T^c}. 
	\]
	We denote by $\sigma_{S,T}$ the permutation written in two line notation as $\tiny\begin{pmatrix}
		S & S^c \\ T & T^c
	\end{pmatrix}$. Using Lemma~\ref{lem:SignofPer}, ${\rm sgn}(\sigma_{T,S}) =  (-1)^{\sum T + \sum S}$. Thus
	\[
	\det(A) e_{[n]} = c_1 \wedge \cdots \wedge c_n =  \sum_{\substack{T\subset[n]\\ |T| = k}} (-1)^{\sum T + \sum S} A_{T,S} A_{T^c,S^c} e_{[n]}.
	\] 
	Therefore $\det(A)  = \sum_{\substack{T\subset[n]\\ |T| = k}} (-1)^{\sum T + \sum S} A_{T,S} A_{T^c,S^c}$ as desired.
\end{proof}

\section{Matrices with a single point fibers}\label{sec:Main}
In this section we characterize $\mathscr{D}(n)$, the set of all matrices with a single point fiber up to diagonal equivalence.
\begin{thm}\label{thm:Main}
	Let $A$ be an $n\times n$ matrix with entries in a field $\F$ with $n\geq 4$. The matrix $A$ belongs to the set $\mathscr{D}(n)$ if and only if  $A$ is irreducible and either 
	\begin{enumerate}
		\item[(a)] $A$ is diagonally equivalent to a symmetric matrix or
		\item[(b)] $A$ does not have any cut, that is ${\rm rk}\left(A[X,X^c]\right)\geq 2$ or ${\rm rk}\left(A[X^c,X]\right)\geq 2$ for all $X \subset [n]$, with $2\leq|X|\leq n-2$.
	\end{enumerate}	 
\end{thm}
\begin{proof}[Proof of ($\Leftarrow$)] Follows from \cite[Theorem~1]{Loewy} and Lemma~\ref{lem:FiberSymIrr}.
\end{proof}	
To prove the other directions, we will use the following lemma repeatedly. The reader is referred to \cite{HartfielLoewy} for the proof.
\begin{lem}\label{lem:NonzeroAdj}[Theorem~1 in \cite{HartfielLoewy}]
	Let $A$ be an $n\times n$ irreducible matrix with entries in a field $\F$. Let $M=(\diag\{x_1,\hdots,x_n\}+A)$ and $G=M^{\rm adj}$, the adjugate matrix of $M$. Then the entries of $G$ are nonzero.
\end{lem}
\begin{lem}\label{lem:CutFactor}
	Let $A$ be an $n\times n$ irreducible matrix with entries in a field $\F$ with $n\geq 4$. Suppose that $A$ has a cut $X\subset [n]$. Let 
	$G = (\diag\{x_1,\hdots,x_n\}+A)^{\rm adj}$. Then for all $i \in X$ and $j \in X^c$ we have
	\[
	G_{ij} = (-1)^ia_ib_j
	\]
	where $a_i \in \F[x_k:k\in X\setminus\{i\}]$, $b_j \in \F[x_k:k\in X^c\setminus\{j\}]$ and both are nonzero polynomials.
\end{lem}
\begin{proof}
	Since $A$ is irreducible and has a cut $X$, then ${\rm rk}(A[X,X^c]) = {\rm rk}(A[X^c,X])=1$. Let $M = (A + \diag\{x_1,\hdots,x_n\})$. We denote by $X_i$ the set $X\setminus \{i\}$. For $S,T\subset [n]$, let $M_{S,T}$ denote the determinant of the submatrix $M[S,T]$ indexed by $S$ on the rows and $T$ on the columns. Since ${\rm rk}(A[X,X^c])=1$, then for each $i\in X$ there exists $a_i\in \F[x_k: k \in X_i]$ such that $M_{X,X_i\cup \{j\}} = c_j a_i$ for all $j \in X^c$ and for some $c_j \in \F$. Now using Theorem~\ref{thm:Laplace} we have:
	\[
	G_{ij}  = \sum_{\substack{T\subset [n]_i\\ |T| = k}} (-1)^{\sum X+\sum T } M_{X,T} M_{(X^c)_j,(T^c)_i} 
	\]
	for $i \in X$ and $j \in X^c$. Notice that for all $T \subset [n]_i$ such that $|T\cap X^c|\geq 2$, the submatrix $M[X,T]$ is not full rank, since ${\rm rk}(A[X,X^c]) = 1$. Thus for all such $T$, $M_{X,T} = 0$. Therefore, $G_{ij}$ can be written as:
	\begin{align*}
		G_{ij}  & = \sum_{\ell\in X^c} (-1)^{ \sum X+\sum X - i + \ell } M_{X,\ell \cup X_i} M_{(X^c)_j,(X^c)_{\ell} }\\
		& = (-1)^i a_i \sum_{\ell\in X^c} (-1)^{\ell} c_\ell M_{(X^c)_j,(X^c)_{\ell}}
	\end{align*}
	using $M_{X,\ell \cup X_i } = c_\ell a_i$. Letting $b_j = \displaystyle\sum_{\ell\in X^c} (-1)^{\ell} c_\ell M_{(X^c)_j,(X^c)_{\ell},  }$, we obtain the desired result. By Lemma~\ref{lem:NonzeroAdj}, $G_{ij} \neq 0$. Therefore, $a_i$ and $b_j$ are both nonzero.
\end{proof}	

\begin{lem}\label{lem:SymEqu}
	Let $A$ be an $n\times n$ irreducible matrix with entries in a field $\F$ with $n\geq 4$. Consider $G= (\diag\{x_1,\hdots,x_n\}+A)^{\rm adj}$. If $A$ has a cut $X\subset [n]$ such that for each $i \in X$ and $j\in X^c$ there exists $\alpha_{ij}\in \F$ satisfying $G_{ij} = \alpha_{ij} G_{ji}$, then $A$ is diagonally similar to a symmetric matrix.
\end{lem}	
\begin{proof}
	Lemma~\ref{lem:CutFactor} implies that for all $i,j \in X$ and $k,\ell\in X^c$ we have $G_{ik}/G_{jk} = G_{i\ell}/G_{j\ell}$ and $G_{ki}/G_{kj} = G_{\ell i}/G_{\ell j}$. Therefore, $\alpha_{ik}/\alpha_{jk} = \alpha_{i\ell}/\alpha_{j\ell}$.
	Let $i, j \in X$. We claim that  $G_{ij} = \alpha_{ik}/\alpha_{jk} G_{ji}$ for any $k \in X^c$. Using Theorem~\ref{thm:Dodgson}, $\Delta_{ij}(f) = G_{ij} G_{ji}$. Then 
	\begin{align*}
	{\rm res}_{x_j}(G_{ik},f) = G_{ij} G_{jk} = \alpha_{jk} G_{ij} G_{kj},
	\end{align*}
	by Theorem~\ref{thm:AV}. Since $G_{ik} = \alpha_{ik} G_{ki}$, then
	\[{\rm res}_{x_j}(G_{ik},f) = {\rm res}_{x_j}(\alpha_{ik} G_{ki},f)=\alpha_{ik}{\rm res}_{x_j}(G_{ki},f)=\alpha_{ik} G_{kj} G_{ji}.\]
	The two equations above now give $G_{ij} = \alpha_{ik}/\alpha_{jk} G_{ji}$. Using similar argument, one can prove that $G_{ij} = \alpha_{kj}/\alpha_{ki} G_{ji}$ for $i,j \in X^c$. Fix $i \in X$ and $k \in X^c$ and let $D$ be the diagonal matrix such that $D_{kk} = 1$, $D_{jj} = \sqrt{\alpha_{jk}}$ if $j \in X$, and $D_{jj} = \sqrt{\alpha_{ik}}/\sqrt{\alpha_{ij}}$ if $j \in X^c$. Since for all $i,j \in X$ and $k,\ell\in X^c$, ${\alpha_{ik}}/{\alpha_{i\ell}}={\alpha_{jk}}/{\alpha_{j\ell}}$, then  $D^{-1} G D$ is a symmetric matrix and so is $D A D^{-1}$ as desired.
\end{proof}

If an irreducible matrix $A$ has a cut $X=[k]$, then Lemma~\ref{lem:CutFactor} states that the matrix $G = (\diag\{x_1,\hdots,x_n\}+A)^{\rm adj}$ can be written as 
\begin{equation*}
	G = \begin{pmatrix} P & R \\ T & Q \end{pmatrix}
\end{equation*} 
where $P$ is a $k\times k$ submatrix, $R_{ij} = (-1)^i a_ib_j$ with $a_i \in \F[x_k: k \in X]$ and $b_j \in \F[x_k:k\in X^c]$, and $T_{ij} = (-1)^j c_i d_j$ with $c_i \in \F[x_k:x_k \in X^c]$ and $d_j \in \F[x_k:k \in X]$. By switching suitable factors between the entries $R_{ij}$ and $T_{ij}$, we get a new matrix $H$. This matrix is the adjugate of the matrix $(\diag\{x_1,\hdots,x_n\}+B)$ where $B$ is a point in the fiber of $A$ as the following lemma shows. We will use this operation later to find a point in the fiber of $A$ that is not diagonally equivalent to $A$.
\begin{lem}\label{lem:Construction}
	Let $A$ be an $n\times n$ irreducible matrix with entries in a field $\F$ with $n\geq 4$. Assume that $A$ has a cut $X=[k]$. Let \[
	G=(\diag\{x_1,\hdots,x_n\}+A)^{\rm adj}=\begin{pmatrix} P & R \\ T & Q \end{pmatrix}
	\]
	where $P$ is a $k\times k$ submatrix, $R_{ij} = (-1)^i a_ib_j$ with $a_i \in \F[x_k: k \in X]$ and $b_j \in \F[x_k:k\in X^c]$, and $T_{ij} = (-1)^j c_i d_j$ with $c_i \in \F[x_k:x_k \in X^c]$ and $d_j \in \F[x_k:k \in X]$. Let $H$ be the matrix defined as
	\begin{equation*}
		H =  \begin{pmatrix} P & U \\ V & Q^T \end{pmatrix}
	\end{equation*}
	with $U_{ij} = (-1)^ia_i c_j$ and $V_{ij} = (-1)^j b_i d_j$. Then $H = (\diag\{x_1,\hdots,x_n\}+B)^{\rm adj}$ for some $n\times n$ matrix $B$ such that $\det(\diag\{x_1,\hdots,x_n\}+A)=\det(\diag\{x_1,\hdots,x_n\}+B)$. 
\end{lem}	
\begin{proof}
	Let $f = \det(\diag\{x_1,\hdots,x_n\}+A)$. Using Theorem~\ref{thm:Dodgson} we have
	\[
	\Delta_{ij}(f) = G_{ij} G_{ji} = (-1)^{i+j}a_i b_j c_j d_i=H_{ij}H_{ji},
	\] 
	for $i \in X$ and $j \in X^c$. Thus, $\Delta_{ij}(f) = H_{ij}H_{ji}$, where $H_{ij}$ and $H_{ji}$ are multiaffine for all $i,j \in [n]$. Using Theorem~\ref{thm:AV}, it remains to show that for all $k$, ${\rm res}_{x_k}(H_{ij},f) = H_{ik}H_{kj}$ for all distinct $i,j \in [n]$. We only consider the case $i\in X$, $j \in X^c$, and $k \in X$, as the other cases follow similarly. Since $b_j \in \F[x_\ell:\ell \in X^c]$ is independent of the variable $x_k$, the resultant ${\rm res}_{x_k}(G_{ij},f) = {\rm res}_{x_k}((-1)^i a_i b_j,f) =(-1)^i b_j {\rm res}_{x_k}(a_i,f)$. Also, using Theorem~\ref{thm:AV} it holds that
	\[
	{\rm res}_{x_k}(G_{ij},f) = G_{ik}G_{kj} = (-1)^k G_{ik} a_k b_j.
	\] 
	This yields ${\rm res}_{x_k}(a_i,f) = (-1)^{i+k}a_k G_{ik}$. Since $c_j \in \F[x_\ell:\ell \in X^c]$ is independent of the variable $x_k$, we obtain
	\[
	{\rm res}_{x_k}(H_{ij},f) = {\rm res}_{x_k}((-1)^i a_i c_j,f) =  (-1)^i c_j {\rm res}_{x_k}(a_i,f) = (-1)^k c_j a_k G_{ik}=H_{kj} G_{ik}.
	\] 
	Since $i,k\in X$, then $  G_{ik}= P_{ik} =H_{ik}$. Therefore ${\rm res}_{x_k}(H_{ij},f) =H_{kj} H_{ik}$. We apply Theorem~\ref{thm:AV} and obtain the existence of a matrix $B$ with $\det(\{x_1,\hdots,x_n\}+B) = f$. 
\end{proof}

\begin{proof}[Proof of Theorem~\ref{thm:Main} ($\Leftarrow$)] First suppose that $A$ is reducible. Then, after permuting some rows and the corresponding columns, $A$ can be written as $\tiny{\begin{pmatrix}
			E & F \\ 0 & G
	\end{pmatrix}}$, with $E$ is a $k\times k$ submatrix. Let $H$ be any $k \times n-k$ matrix such that $H_{ij} = 0$ if and only if $F_{ij} \neq 0$ and $H_{ij} = 1$ otherwise. Let $B = \tiny{\begin{pmatrix}
			E & H \\ 0 & G
	\end{pmatrix}}$. If $S \subseteq [k]$, then $A_S = E_S = B_S$, and similarly if $S\subseteq [k]^c$. Otherwise, if $S= S_1 \cup S_2$ with $S_1 \subset [k]$ and $S_2 \subset [k]^c$, then $A_S = E_{S_1} G_{S_2} = B_S$. Therefore, $A$ and $B$ have the same principal minors. Notice that $B$ is not diagonally similar to $A$. Therefore, $A \notin \mathscr{D}(n)$.
	
	Now assume that $A$ is irreducible, $A$ is not diagonally equivalent to a symmetric matrix, and $A$ has a cut $X=[k]\subset [n]$. Notice that ${\rm rk}(A[X,X^c])$ and ${\rm rk}(A[X^c,X])$ are nonzero since $A$ is irreducible. Thus ${\rm rk}\left(A[X,X^c]\right)={\rm rk}\left(A[X^c,X]\right)= 1$. Let $M = {\rm diag}\{x_1,\hdots,x_n\} + A$, $f = \det\left(M\right)$,   and $G = M^{\rm adj}$. Using Lemma~\ref{lem:CutFactor}, for all $i \in X$ and $j \in X^c$ the entry $G_{ij} = (-1)^i a_i b_j$ for some $a_i \in \F[x_k: k\in X_i]$ and $b_j \in \F[x_k:k \in (X^c)_j]$. By Lemma~\ref{lem:NonzeroAdj}, $G_{ij} \neq 0$. therefore, neither $a_i$ nor $b_j$ can be zero and so the matrix $G$ can be written as a block matrix
	\begin{equation*}
		G = \begin{pmatrix} P & R \\ T & Q \end{pmatrix}
	\end{equation*} 
	where $P$ is a $k\times k$ submatrix, $R_{ij} = (-1)^i a_ib_j$ and $T_{ij} = (-1)^j c_i d_j$ with $a_i \in \F[x_k: k \in X_i]$, $b_j \in \F[x_k:k\in (X^c)_j]$, $c_i \in \F[x_k:x_k \in (X^c)_i]$, and $d_j \in \F[x_k:k \in X_j]$.  Consider the matrix $H$ defined by
	\begin{equation*}
		H =  \begin{pmatrix} P & U \\ V & Q^T \end{pmatrix}
	\end{equation*}
	with $U_{ij} = (-1)^ia_i c_j$ and $V_{ij} = (-1)^j b_i d_j$. We infer that $H=(\diag\{x_1,\hdots,x_n\}+B)^{\rm adj}$ where $B$ has the same principal minors as $A$, by Lemma~\ref{lem:CutFactor}.\\
	We still need to check that $B$ is not diagonally similar to $A$. We present a proof by contradiction. Indeed, assume there is a diagonal matrix $D$ such that $B=DAD^{-1}$. Then $H = DGD^{-1}$. Let $D = \diag\{\lambda_1,\hdots,\lambda_n\}$ with $\lambda_i \in \F$. Then $H_{ij} = \lambda_{j}/\lambda_i G_{ij}$ for all $i,j \in[n]$. In particular, $H_{ij} =(-1)^ia_i c_j = \lambda_j/\lambda_i (-1)^i a_i b_j$ for all $i \in X$ and $j \in X^c$. Hence, $b_j = \alpha_j c_j$ for some $\alpha_j \in \F$ and for $j \in X^c$. If this is the case, then we switch $a_i$ and $d_i$ in $G$ instead of $b_j$ and $c_j$. That is we let $H=\tiny\begin{pmatrix}
		P^T & U \\ V & Q
	\end{pmatrix} $ such that $U_{ij} = (-1)^i d_i b_j$ and $V_{ij} = (-1)^i a_i c_j$. Using the same argument as the one above, we obtain a matrix $C$ with the same principal minors as those of $A$. If $C$ is also diagonally congruent to $A$, then we get $a_i = \beta_i d_i$ for some $\beta_i \in \F$ and for all $i \in X$. In this case, by Lemma~\ref{lem:SymEqu}, $A$ is diagonally equivalent to a symmetric matrix and this gives the desired contradiction.
\end{proof}

\begin{example}
	Consider the $4\times 4$ matrix $A$ 
	\[
	A = \tiny{\begin{pmatrix}
			2 & -1 & 1 & -2 \\
			1 & 1 & -3 & 6 \\
			1 & 2 & 1 & 1 \\
			-1 & -2 & 2 & -1 \\
			
	\end{pmatrix}}.
	\]
	Notice that $A$ has a cut at $X=\{1,2\}$ since ${\rm rk}(A[X,X^c]) = {\rm rk}(A[X^c,X]) = 1$. We compute $G = (A+\diag\{x_1,\hdots,x_4\})^{\rm adj}$ and obtain
	\[
	G = \tiny{\begin{pmatrix}
			p_1 & ({x_3}+3) ({x_4}+3) & -({x_2}-2) ({x_4}+3) & ({x_2}-2) (2 {x_3}+3) \\
			-15 - 5 x_3 - 4 x_4 - x_3 x_4 & p_2 & (3 {x_1}+7) ({x_4}+3) & -(3 {x_1}+7) (2 {x_3}+3) \\
			-({x_2}-1) {x_4} & -(2 {x_1}+5) {x_4} & p_3 & -{x_1} {x_2}+11 {x_1}-4 {x_2}+29 \\
			({x_2}-1) ({x_3}+3) & (2 {x_1}+5) ({x_3}+3) & -(2 {x_1}+5) ({x_2}-2) & p_4 \\
	\end{pmatrix}}
	\]
	with $p_i = G_{ii}$. If we switch the factors $(x_2-2)$ and $-(3x_1+7)$ in $G[X,X^c]$ with the factors $(x_2-1)$ and $(2x_1+5)$ in $G[X^c,X]$ respectively and we transpose the submatrix $A[X,X]$,  we get
	\[
	H = \tiny{\begin{pmatrix}
			p_1 & -{x_3}{x_4}-5 {x_3}-4 {x_4}-15 & -({x_2}-1) ({x_4}+3) & ({x_2}-1) (2 {x_3}+3) \\
			({x_3}+3) ({x_4}+3) & p_2 & -(2 {x_1}+5) ({x_4}+3) & (2 {x_1}+5) (2 {x_3}+3) \\
			-({x_2}-2) {x_4} & (3 {x_1}+7) {x_4} & p_3 & -{x_1} {x_2}+11 {x_1}-4 {x_2}+29 \\
			({x_2}-2) ({x_3}+3) & -(3 {x_1}+7) ({x_3}+3) & -(2 {x_1}+5) ({x_2}-2) & p_4 \\
	\end{pmatrix}}
	\]
	which is the adjugate of the matrix $(B+\diag\{x_1,\hdots,x_n\})$ with 
	\[
	B = \tiny{\begin{pmatrix} 
			2 & 1 & 1 & -2 \\
			-1 & 1 & 2 & -4 \\
			1 & -3 & 1 & 1 \\
			-1 & 3 & 2 & -1
	\end{pmatrix}}.
	\]
	By computing the determinantal polynomial of $A$, $f_A = \det(\diag\{x_1,\hdots,x_n\}+A)$, and that of $B$, $f_B = \det(\diag\{x_1,\hdots,x_n\}+B)$, we see that $A$ and $B$ have the same principal minors. Simple computations show that $B$ is not diagonally equivalent to $A$.
\end{example}	


\section{The Structure of a matrix from the factorization of its determinantal polynomial}\label{sec:StructureThm}
In this section, we relate the irreducibility of a matrix $A$ to the irreducibility of its determinantal polynomial $f_A = \det(\diag\{x_1,\hdots,x_n\}+A)$. Moreover, we use the factors of this determinantal polynomial in order to study the structure of the matrix.
\begin{lem}\label{lem:IrrMatrixIrrPoly}
	Let $A$ be an $n\times n$ matrix with entries in a field $\F$. Let $f$ be its determinantal polynomial
	\begin{equation*}
		f = \det\left({\rm diag}\{x_1,x_2,\hdots,x_n\} + A\right).
	\end{equation*}
	Then $A$ is irreducible if and only if $f$ is irreducible.
\end{lem}
\begin{proof}
	$(\Longleftarrow)$ Suppose $A$ is reducible. Then $A$ can be written as a block upper triangular matrix, possibly after permuting some rows and the corresponding columns. Consequently $f$ factors into at least two factors.\\
	$(\Longrightarrow)$ Suppose $A$ is irreducible. Let $M = (A+\diag\{x_1,\hdots,x_n\})$ and $G=M^{\rm adj}$. Then using Lemma~\ref{lem:NonzeroAdj}, we have $G_{ij} \neq 0$ for all $i, j \in [n]$. Thus, Theorem~\ref{thm:Dodgson} implies that $\Delta_{ij}(f) = G_{ij} G_{ji} \neq 0$. In view of Lemma~\ref{lem:ZeroDelta}, we conclude that $f$ is irreducible.
\end{proof}

\begin{thm}\label{thm:StructureThm}
	Let $A$ be an $n\times n$ matrix with entries in a field $\F$. Let $f$ be the polynomial defined by
	\begin{equation*}
		f = \det\left({\rm diag}\{x_1,x_2,\hdots,x_n\} + A\right).
	\end{equation*}
	Then $f$ factors as $f_1\cdots f_s$ such that each $f_i$ is irreducible and $f_i \in \F[x_k:k\in T_i]$ if and only if $A$ can be written, after permuting some rows and the corresponding columns, as a block upper triangular matrix with $s$ diagonal blocks $A_1,\hdots,A_s$ such that each $A_i$ is irreducible and $f_i = \det(A_i+\diag\{x_k:k\in T_i\})$.
	
\end{thm}
\begin{proof}
	$\Longleftarrow$ Suppose that $A$ can be written as a block upper triangular matrix with $s$ irreducible diagonal blocks. Then $f=\prod_{i=1}^s \det(A_i+\diag\{x_k:k\in T_i\})= \prod_{i=1}^sf_i$. Each $f_i$ is irreducible in view of Lemma~\ref{lem:IrrMatrixIrrPoly}.\\
	$\Longrightarrow$ Suppose that $f$ is reducible and factors as $f_1\cdots f_s$. We proceed by induction on $s$. The case $s=1$ follows from Lemma~\ref{lem:IrrMatrixIrrPoly}. For $s>1$, Lemma~\ref{lem:IrrMatrixIrrPoly} implies that $A$ is reducible and so it can be written as a block upper triangular matrix with two diagonal blocks $A_1$ and $A_2$. Without loss of generality, $\det(A_1) = f_1\cdots f_t$ for $t<s$. Thus, $A_1$ satisfies all the inductive hypothesis and it can be written as a block upper triangular matrix with $t$ diagonal irreducible blocks, after permuting some rows and the corresponding columns of $A_1$. This permutation will not affect $A_2$. The same holds for $A_2$ and we write it as a block upper triangular matrix with $s-t+1$ diagonal blocks. This proves that $A$ can be written as a block upper triangular matrix with $s$ diagonal blocks with associated determinantal polynomials that equal the factors of $f$.
\end{proof}
\begin{remark}
	Notice that Theorem~\ref{thm:StructureThm} implies that any matrix can be rearranged by permuting some rows and the corresponding columns to form a block upper triangular matrix with irreducible diagonal blocks. This fact was proved by Brualdi and Ryser \cite[Theorem~3.2.4]{brualdi_ryser_1991}. They gave a purely combinatorial proof. From a graph theory perspective, the theorem can be stated as follows: \\
	To each $n\times n$ matrix $A=(a_{ij})_{1\leq1 i,j\leq n}$, we associate a direct graph $G$ with $n$ vertices $v_1,\hdots,v_n$. There is a directed edge $(i,j)$ from a vertex $v_i$ to a vertex $v_j$ if and only if $a_{ij} \neq 0$. Suppose that $G$ is composed of strongly connected components $C_1,\hdots,C_s$. By contracting each component $C_i$ into a single vertex $V_i$, we obtain a new graph $H$ with $s$ vertices $V_1,...,V_s$. Directed edges are added to $H$ such that there is a directed edge from $V_i$ to $V_j$ if and only if there exists a directed edge from a vertex in $C_i$ to a vertex in $C_j$. Thus, $H$ is a directed acyclic graph, allowing for a topological ordering of its vertices. This topological ordering corresponds to the diagonal irreducible blocks.\\
	This gives a very nice connection between the strongly connected components of the graph $G$ that corresponds to a matrix $A$ and the factors of $f_A = \det(\diag\{x_1,\hdots,x_n\}+A)$.
\end{remark}
\begin{cor}\label{cor:GeneralFiber}
	Let $A$ be an $n\times n$ matrix with entries in the field $\F$. Let $f$ be the polynomial defined by
	\begin{equation*}
		f = \det\left({\rm diag}\{x_1,x_2,\hdots,x_n\} + A\right).
	\end{equation*}
	Suppose that $f$ factors as $f_1\cdots f_s$ such that each $f_i$ is irreducible and $f_i \in \F[x_k:k\in T_i]$. Then the fiber of $A$ consists of all matrices that can be written as block upper triangular matrices with $s$ diagonal irreducible blocks $A_1,\hdots,A_s$ with $\det(A_i+\diag\{x_k:k\in [T_i]\})=f_i$ for each $i$.
\end{cor}

\begin{example}
	Consider the $6\times 6$ matrix $A$
	\[
	A = \tiny{\left(
		\begin{array}{cccccc}
			1 & -3 & 3 & -2 & -1 & 2 \\
			0 & -3 & 5 & 1 & 0 & 2 \\
			0 & 0 & 4 & 0 & 0 & -4 \\
			0 & 1 & 2 & 1 & 0 & 5 \\
			1 & 0 & -1 & 6 & 2 & 4 \\
			0 & 0 & 2 & 0 & 0 & 3 \\
		\end{array}
		\right)}.
	\]
	Then the determinantal polynomial of $A$ is given by
	\[
	f_A = \det(\diag\{x_1,\hdots,x_n\}+A)=(x_1x_5+2 x_1+x_5+3) (x_2 x_4+x_2-3x_4-4) (x_3 x_6+3 x_3+4 x_6+20).
	\]
	According to Theorem~\ref{thm:StructureThm}, this matrix can be written as a block upper triangular matrix with three irreducible diagonal blocks. Since the fiber of a $2\times 2$ irreducible matrix is a point up to diagonal equivalence, then Corollary~\ref{cor:GeneralFiber} tells us that the fiber of $A$ consists of matrices that can be written in the following form up to permuting the diagonal blocks and up to diagonal equivalence:
	\[
	\tiny{
	\begin{pmatrix}
		\begin{matrix}
			2 & 1\\ 
			-1 & 1 \end{matrix} & \mbox{\normalfont\large $\star$ }& \mbox{\normalfont\large $\star$ }\\
		\mbox{\normalfont\large 0} & \begin{matrix} 1 & 1 \\ 1 & -3 \end{matrix} & \mbox{\normalfont\large $\star$} \\
		\mbox{\normalfont\large 0} & \mbox{\normalfont\large 0} & \begin{matrix} 3 & 2\\-4 & 4\end{matrix}
	\end{pmatrix}}
	\]
	where $\star$ can be replaced with any $2\times 2$ matrix.
\end{example}
\begin{cor}
	Let $A$ be an $n\times n$ irreducible matrix with entries in a field $\F$ with $n\geq 4$. The fiber of $A$ consists of irreducible matrices.
\end{cor}
\begin{proof}
	Let $f = \det(A + \diag\{x_1,\hdots,x_n\})$. By Lemma~\ref{lem:IrrMatrixIrrPoly}, $f$ is irreducible. Let $B$ be any matrix in the fiber of $A$, then $\det(B+\diag\{x_1,\hdots,x_n\})=f$. Hence this later determinant is irreducible. By Lemma~\ref{lem:IrrMatrixIrrPoly}, $B$ is irreducible.
\end{proof}

\section{The fibers of Symmetric Matrices}\label{sec:FiberSym}

\begin{lem}\label{lem:FiberSymIrr}
	Let $A$ be an $n\times n$ symmetric matrix with $n\geq 4$ and with entries in the field $\F$. Then $A$ belongs to the set $\mathscr{D}(n)$ if and only if $A$ is irreducible.
\end{lem}
\begin{proof}
	$\Longrightarrow$ If $A$ is reducible and symmetric, then $A$ can be written as a block diagonal matrix with diagonal blocks that are symmetric. Let $B$ be any block upper triangular matrix with the same diagonal blocks where we replace all the upper blocks of the matrix with nonzero entries. Direct computations show that $A$ and $B$ have equal principal minors, but $B$ is not diagonally equivalent to $A$. Therefore $A \notin \mathscr{D}(n)$. \\
	$\Longleftarrow$ Let $B $ be a matrix with the same principal minors as $A$. Then
	\[
	f_A = \det\left({\rm diag}\{x_1,x_2,\hdots,x_n\} + A\right)=\det\left({\rm diag}\{x_1,x_2,\hdots,x_n\} + B\right)=f_B.
	\] 
	Then $\Delta_{ij}(f_A) = \Delta_{ij}(f_B)$. Notice that $G = (\diag\{x_1,\hdots,x_n\}+A)^{\rm adj}$ is also symmetric. Using Theorem~\ref{thm:Dodgson}, $\Delta_{ij}(f_A)=G_{ij}^2$. Let $H= (\diag\{x_1,\hdots,x_n\}+B)^{\rm adj}$. Then $\Delta_{ij}(f)$ equals $G_{ij}^2 = H_{ij} H_{ji}$. Since $G_{ij}$ and $H_{ij}$ are multiaffine, there exists $\alpha_{ij}\in\F$ such that $H_{ij} = \alpha_{ij} G_{ij}$ and $H_{ji} = \frac{1}{\alpha_{ij}} G_{ij}$ for $1\leq i < j \leq n$. Since $A$ is irreducible, then so is $f$. By Lemma~\ref{lem:ZeroDelta}, $\Delta_{ij}(f) \neq 0$ and so $\alpha_{ij} \neq 0$. Using Theorem~\ref{thm:AV}, ${\rm res}_{x_k}(H_{ij},f) = H_{ik} H_{kj}$. Then $\alpha_{ij}= \alpha_{ik} \frac{1}{\alpha_{jk}}$ for $1 \leq i < j<k \leq n$. Let $D = {\rm diag}\{1,\alpha_{12},\hdots,\alpha_{1n}\}$, then $B = D A D^{-1}$ and $A \in \mathscr{D}(n)$.
\end{proof}

\begin{thm}\label{thm:FiberSym}
	Let $A$ be an $n\times n$ symmetric matrix with entries in a field $\F$ and with $n\geq 4$. Then either
	\begin{enumerate}
		\item[(a)] The matrix $A$ is irreducible and the fiber of $A$ consists of a point up to diagonal equivalence or
		\item[(b)] the matrix $A$ can be written as a completely reducible matrix with $s$ irreducible diagonal blocks in which case the fiber of $A$ consists of all matrices that can be written as block upper triangle matrices with $s$ irreducible diagonal blocks that are diagonally equivalent to the diagonal blocks of $A$, after permuting some rows and the corresponding columns.
	\end{enumerate}  
\end{thm}
\begin{proof}
	Part $(a)$ of the theorem is Lemma~\ref{lem:FiberSymIrr}. For $(b)$, assume that $A$ is completely reducible with $s$ irreducible diagonal blocks $A_1,\hdots,A_s$ indexed on the rows and columns by $T_1,\hdots,T_s$ respectively with $T_i \subset [n]$ and let $B$ be a matrix in the fiber of $A$ under the principal minor map. Using Theorem~\ref{thm:StructureThm}, $f_A = \det(\diag{x_1,\hdots,x_n}+A)$ is reducible and can be written as $f_1\cdots f_s$ where each $f_i = \det(A_i + \diag\{x_k:k \in T_i\})$ is irreducible. Since $A$ and $B$ have equal determinantal polynomials $f_A = f_B = \det(\diag{x_1,\hdots,x_n}+B)$, then , in view of Theorem\ref{thm:StructureThm}, $B$ can be written as a block upper triangular matrix with $s$ irreducible diagonal blocks $B_i$ with $1\leq i \leq s$ and such that $f_i = \det(B_i + \diag\{x_k:k \in T_i\})$. Thus $B_i$ and $A_i$ have the same principal minors for each $i$. Since $B_i$ is symmetric and irreducible, then by Lemma~\ref{lem:FiberSymIrr} $A_i$ is diagonally equivalent to $B_i$ as desired.
\end{proof}

\section{The fibers of Hermitian matrices and Determinantal Stable Polynomials}\label{sec:FiberHer}
In this section we restrict to the field of complex numbers $\C$.

\begin{lem}\label{lem:FiberHerIrr}
	Let $A$ be an $n\times n$ complex matrix with $n\geq 4$. The fiber of $A$ consists of irreducible Hermitian matrices up to diagonal equivalence if and only if $A$ is diagonally equivalent to an irreducible Hermitian matrix.
\end{lem}
\begin{proof}
	$\Longleftarrow$ Suppose that $A$ is Hermitian. Let $f=\det(\diag\{x_1,\hdots,x_n\}+A)$. In view of \cite[Theorem 3.2]{BBL09}, $f$ is real stable. Moreover, by Lemma~\ref{lem:IrrMatrixIrrPoly}, $f$ is irreducible since $A$ is irreducible. Let $B$ be a matrix in the fiber of $A$. Then $\det(B+\diag\{x_1,\hdots,x_n\}) = f$. By \cite[Theorem $6.4$]{AV22}, $B$ is diagonally equivalent to a Hermitian irreducible matrix.\\
	$\Longrightarrow$ Let $B$ be a matrix in the fiber of $A$. Then $B$ is diagonally equivalent to an irreducible Hermitian matrix. By the above argument, $A$ is diagonally equivalent to an irreducible Hermitian matrix.
\end{proof}
\begin{thm}\label{thm:FiberHer}
	Let $A$ be an $n\times n$ Hermitian matrix with $n\geq 4$. Then either 
	\begin{enumerate}
		\item[(a)] The matrix $A$ is irreducible and the fiber of $A$ consists of Hermitian irreducible matrices up to diagonal equivalence or
		\item[(b)] The matrix  $A$ can be written as a completely reducible matrix with $s$ irreducible diagonal blocks, after permuting some rows and the corresponding columns, in which case the fiber of $A$ consists exactly of matrices that can be written as block upper triangular matrices with $s$ diagonal  blocks that are irreducible and diagonally equivalent to Hermitian irreducible matrices and that have determinantal polynomials equal to the factors of $f$.
	\end{enumerate}
\end{thm}
\begin{proof}
	Similar proof as that of Theorem~\ref{thm:FiberSym}
\end{proof}
Notice that, unlike the symmetric case, the diagonal blocks of a matrix that lies in the fiber of a Hermitian matrix $A$ need not be diagonally equivalent to the diagonal blocks of $A$. See for instance \cite[Example 4.8]{AV22}. 

Now we are ready to answer Question~\ref{ques:StableMatrices} by Borcea, Br\"{a}nd\'{e}n, and Liggett \cite[Question~3.4]{BBL09} about real stable matrices. A real polynomial $f\in \R[x_1,\hdots,x_n]$ is \emph{stable} if it has no zeros with strictly positive imaginary parts. A real stable matrix $A$ is an $n\times n$ complex matrix with real stable determinantal polynomial i.e $f_A = \det(\diag\{x_1,\hdots,x_n\}+A)$ is real stable. 
\begin{cor}\label{cor:StableMatrix}
	Let $A$ be an $n\times n$ complex matrix such that $f = \det(\{x_1,\hdots,x_n\}+A)$ is real stable. Then 
	\begin{enumerate}
		\item[(a)] If $f$ is irreducible, then $A$ is diagonally equivalent to an irreducible Hermitian matrix with equal principal minors.
		\item[(b)] If $f$ factors as $ f_1\cdots f_s$ where each factor $f_k$ is irreducible, then $A$ can be written, after permuting some rows and the corresponding columns, as a block upper triangular matrix with $s$ diagonal blocks that are irreducible and diagonally equivalent to Hermitian irreducible matrices that have determinantal polynomials equal to the irreducible factors of $f$.
	\end{enumerate} 
\end{cor}
\begin{proof}
	Part $(a)$ of the corollary follows from \cite[Theorem $6.6$]{AV22} and Lemma~\ref{lem:IrrMatrixIrrPoly}. For part $(b)$, Theorem~\ref{thm:StructureThm} implies that $A$ can be written as a block upper triangular matrix with $s$ diagonal irreducible blocks that have determinantal polynomials equal to the factors of $f$. Since each $f_i$ is stable and irreducible, then by \cite[Theorem~6.6]{AV22} and Lemma~\ref{lem:IrrMatrixIrrPoly} each diagonal block is diagonally equivalent to an irreducible Hermitian matrix as desired.
\end{proof}

\bibliographystyle{plain}

\end{document}